\documentclass[12pt,reqno]{amsart}
\usepackage{amssymb}

\usepackage{tikz-cd}

\newtheorem{theorem}{Theorem}[section]
\newtheorem{lemma}[theorem]{Lemma}
\newtheorem{proposition}[theorem]{Proposition}
\newtheorem{corollary}[theorem]{Corollary}

\theoremstyle{definition}
\newtheorem{definition}[theorem]{Definition}

\numberwithin{equation}{section}

\begin{document}

\baselineskip=15.5pt

\title[Semisimple algebras of vector fields on ${\mathbb C}^N$]{Semisimple
algebras of vector fields on ${\mathbb C}^N$ of maximal rank}

\author[H. Azad]{Hassan Azad}

\address{Abdus Salam School of Mathematical Sciences, GCU, Lahore 54600, Pakistan}

\email{hassan.azad@sms.edu.pk}

\author[I. Biswas]{Indranil Biswas}

\address{Department of Mathematics, Shiv Nadar University, NH91, Tehsil Dadri,
Greater Noida, Uttar Pradesh 201314, India}

\email{indranil.biswas@snu.edu.in, indranil29@gmail.com}

\author[F. M. Mahomed]{Fazal M. Mahomed}

\address{DSI-NRF Centre of Excellence in Mathematical and Statistical Sciences, School of Computer Science
and Applied Mathematics, University of the Witwatersrand, Johannesburg, Wits 2050, South Africa}

\email{Fazal.Mahomed@wits.edu.za}

\subjclass[2000]{17B66, 32M25, 57R25}

\keywords{Vector field, semisimple Lie algebra, Levi decomposition.}

\date{}

\begin{abstract}
A local classification of semisimple Lie algebras of vector fields on ${\mathbb C}^N$ that have a Cartan subalgebra
of dimension $N$ is given.
The proof uses basic representation theory and the local canonical form
of semisimple Lie algebras of vector fields.
\end{abstract}

\maketitle

\section{Introduction}\label{se1}

All vector fields considered in this paper are assumed to be complex analytic.
By a Lie algebra of vector fields on ${\mathbb C}^N$ Lie meant a Lie algebra of
analytic vector fields defined on some connected open subset ${\mathcal U}$ of ${\mathbb C}^N$. He considered two such Lie algebras
$L_1$ and $L_2$, defined on open subsets ${\mathcal U}_1$ and ${\mathcal U}_2$ respectively, to be equivalent if there is an
open set ${\mathcal U}\, \subset\, {\mathcal U}_1$
and an analytic map $\phi\, :\, {\mathcal U}\, \longrightarrow\, {\mathcal U}_2$, which is a diffeomorphism
between ${\mathcal U}$ and $\phi({\mathcal U})$ and $\phi_*(L_1\big\vert_{\mathcal U})\,=\,
L_2\big\vert_{\phi(\mathcal U)}$.
This is the same as saying that all Lie algebras obtained from a given Lie algebra $L$ of vector fields by a local change of coordinates are considered to be equivalent.
Lie gave, in this sense, a classification of all finite dimensional
Lie algebras of vector fields on ${\mathbb C}^2$ \cite{Li}. This was later extended to vector fields on the real plane in \cite{GKO}.

In this paper we give a local classification of all finite dimensional complex semisimple Lie algebras of vector fields on
${\mathbb C}^N$ that have a Cartan subalgebra of dimension $N$. The main results proved here is the following:

\begin{theorem}\label{thm1}
If $\mathcal G$ is a finite dimensional semisimple Lie algebra of vector fields on ${\mathbb C}^N$ which has a
Cartan subalgebra of dimension $N$, then the simple factors of $\mathcal G$ must be of type $A_{\ell_i}$, $1\, \leq\, i\,\leq
\, d$, such that $\sum_{i=1}^d \ell_i \,=\, N$.
\end{theorem}

The proof of Theorem \ref{thm1} uses the local canonical form of semisimple Lie algebras of vector fields on ${\mathbb R}^N$ given in \cite[Theorem 2.1]{ABM}.
This result of \cite{ABM} states that if a semisimple Lie algebra of vector fields on ${\mathbb R}^N$ has a split Cartan subalgebra then there are
local coordinates $x_1,\, \cdots,\, x_N$ with respect to which the Cartan subalgebra is generated by $\partial_{x_i}$, $1\, \leq\, i\, \leq\, r$, where $r$
is the dimension of the Cartan subalgebra. If we have a complex semisimple Lie algebra of vector fields on ${\mathbb C}^N$ then all Cartan subalgebras
of a complex semisimple Lie algebra are split. Thus we see, by repeating the arguments of \cite[Theorem 2.1]{ABM}, that there are local analytic coordinates
$x_1,\, \cdots,\, x_N$ with respect to which the Cartan subalgebra is generated by $\partial_{x_i}$, $1\, \leq\, i\, \leq\, r$, where $r$
is the dimension, over $\mathbb C$, of a given Cartan subalgebra.
We refer the reader to \cite{Ki} for results used from representation theory and root systems in the proof of Theorem \ref{thm1}.

After seeing a preprint of this paper, V. Popov sent us Vinberg's paper \cite{Vi}, where a similar result is proved for
algebraic groups over an algebraically closed field of characteristic zero. In this paper we work with local analytic vector
fields and do not assume that the vector fields are complete. The result of Vinberg is therefore
a consequence of Theorem \ref{thm1}.

\section{Structure of Lie algebra of vector fields}

\begin{lemma}\label{lem2.1}
Let $\mathcal L$ and ${\mathcal L}_1$ be the Lie algebras of analytic vector fields on ${\mathbb C}^{k+m}$ of the form
$$
\sum_{i=1}^{k+m} f_i(x_1,\,\cdots ,\, x_k) \partial_{x_i} \ \ \ and \ \ \
\sum_{i=1}^{k} f_i(x_1,\,\cdots ,\, x_k) \partial_{x_i}
$$
respectively, where $f_i$ are analytic functions of $k$ variables. Then the linear map
$$
\Pi\,\,:\,\, {\mathcal L}\,\, \longrightarrow\,\, {\mathcal L}_1,\ \ \,
\sum_{i=1}^{k+m} f_i(x_1,\cdots , x_k) \partial_{x_i} \,\longmapsto\,
\sum_{i=1}^{k} f_i(x_1,\cdots , x_k) \partial_{x_i}
$$
is a homomorphism of Lie algebras.
\end{lemma}

\begin{proof}
For vector fields $X$ and $Y$, and a smooth function $h$, we have
$$
[X,\, h\cdot Y]\,\,=\,\, h\cdot [X,\, Y]+ X(h)\cdot Y.
$$
Also, $\partial_{x_j} f(x_1,\,\cdots ,\, x_k)\,=\, 0$ for $k+1\, \leq\, j\, \leq\, k+m$.
The lemma follows from these.
\end{proof}

\begin{definition}
Let $L$ be a Lie algebra of vector fields defined on an open subset $\mathcal{U}$ $\subset$ ${\mathbb C}^N$. The rank of $L$ is
$${\rm Max}_{p \in \mathcal{U}} \dim \{ X(p) \ \big\vert \ X \,\in\, L\}.$$
\end{definition}

In general the rank and dimension are different: The abelian Lie algebra 
$$\langle \partial_x,\, y\partial_x,\, \cdots,\, y^{n}\partial_x \rangle$$
has rank 1 and dimension $n+1$.

For a proof of the following basic fact see \cite{Lee}.

\begin{theorem}\label{thm 2.3}
If $X_1,\,\cdots ,\, X_r$ are commuting vector fields defined on an open subset $\mathcal{U}\, \subset\, {\mathbb C}^N$ and $p \,\in\, \mathcal{U}$
a point such that $X_{1}(p),\,\cdots
,\, X_{r}(p)$ are linearly independent, then there are local coordinates $x_1, \,\cdots , \, x_N$ defined near $p$ such that $X_{i} = \partial_{x_i}$ ($1 \leq i \leq r).$
\end{theorem}

Coordinates as in Theorem \ref{thm 2.3} are called \textit{canonical coordinates} for the commuting fields $X_1, \,\cdots , \,X_r$.

We will use Lemma \ref{lem2.1} and Theorem \ref{thm 2.3} to derive the local canonical form of Lie algebras of type $A_{1}\times A_{1} \times \cdots \times A_{1}$.
This is an essential ingredient for the proof of Theorem \ref{thm1}. 

\textbf{Notation}: If a Lie algebra $\mathfrak g$ is generated as a Lie algebra by $X_{1},\,\cdots ,\, X_{n}$, we write
${\mathfrak g}\,=\,\langle X_{1}, \,\cdots ,\, X_{n} \rangle$.

\begin{proposition}\label{prop2.4}
If $L \,=\, \stackrel{N{\rm -}times}{\overbrace{A_{1}\times A_{1} \times \cdots \times A_{1}}}$ is a Lie algebra of vector fields
on ${\mathbb C}^N$ then there are local coordinates $x_{1}, \,\cdots ,\, x_{N}$ in which 
$$L\ =\ \langle \exp (x_1)\partial_{x_1}, \exp (-x_1)\partial_{x_1} \rangle \times \cdots \times \langle \exp ({x_N})
\partial_{x_N},\, \exp(-x_N)\partial_{x_N} \rangle .$$
\end{proposition}

\begin{proof}
The Lie algebra $L$ has generators $X_i$, $Y_i$ ($1 \leq i \leq N$) such that if we set $H_{i} = [X_{i}, Y_{i}]$, then
\begin{equation}\label{ec}
[H_{i}, \, X_{i}] \,=\, X_i, \ \, [H_{i},\, Y_{i}] \, =\, -Y_{i}
\end{equation}
for all $1\, \leq\, i\, \leq\, N$ and $[X_{i},\, X_{j}] \,=\, 0 \,=\, [Y_{i},\, Y_{j}]$ for all $1\, \leq\, i,\, j
\, \leq\, N$. The subalgebra $\langle H_{1}, \,\cdots,\, H_{N} \rangle$ is a Cartan subalgebra of $L$.
By Theorem \ref{thm 2.3} there are local coordinates $x_{1}, \,\cdots ,\, x_{N}$ such that $H_{i}\,=\, \partial_{x_i}$
(see also \cite[Theorem 2.1]{ABM}). The eigenfields for $\partial_{x_i}$ for eigenvalue $\lambda$
are $\exp (\lambda x_{i})(f_{1}\partial_{x_1}+ \ldots + f_{N}\partial_{x_N})$, where $f_{1}, \,\cdots , \,f_{N}$ are functions independent of $x_{i}$. As $[H_{j}, X_{i}] = 0 = [H_{j}, Y_{i}]$ for $i \neq j$, we see that $X_{i} \,=\, \exp (x_{i})(\lambda_{i1}\partial_{x_1} + \cdots +\lambda_{iN}\partial_{x_N})$,
 $Y_{i} \,=\, \exp (-x_{i})(\mu_{i1}\partial_{x_1} + \cdots +\mu_{iN}\partial_{x_N})$, where $\lambda_{ij}$ and $\mu_{ij}$ are constants.

We want to show that $\lambda_{ii}$ and $\mu_{ii}$ are not zero for all $1\, \leq\, i\,\leq\, N$
and $\lambda_{ij} \,=\, 0 \,=\, \mu_{ij}$ if $i \,\neq\, j$. To show that
$\lambda_{ii}$ and $\mu_{ii}$ are not zero for all $1\, \leq\, i\,\leq\, N$, it suffices to show this for $i\,=\,1$ ---
as this amounts to relabeling of variables.

We first show that $\lambda_{11}\,\not=\, 0$.

Suppose $\lambda_{11} = 0$. Then $$X_{1}\, =\, \exp (x_{1})(\lambda_{12}\partial_{x_2} + \cdots +\lambda_{1N}\partial_{x_N})\ \ \text{ and }\ 
\ Y_{1}\, =\, \exp (-x_{1})(\mu_{11}\partial_{x_1} + \cdots +\mu_{1N}\partial_{x_N}).$$ Using the formula 
\begin{equation}\label{1}
[\exp (\chi)U,\, \exp (\psi)V] \,=\, \exp (\chi + \psi)(U(\psi)V - V(\chi)U + [U,\, V]),
\end{equation}
where $\chi$ and $\psi$ are analytic functions, we see, using that $\lambda_{ij}$, $\mu_{ij}$ are constants, that
$H_{1}\, =\, [X_{1},\, Y_{1}] \,=\, -\mu_{11}(\lambda_{12}\partial_{x_2} + \cdots + \lambda_{1N}\partial_{x_N})$. Thus
$[H_{1},\, X_{1}] \,=\, 0$; but this contradicts \eqref{ec}.

Therefore $\lambda_{11} \,\neq\, 0$. Similarly $\mu_{11} \,\neq\, 0$. As mentioned above, this means that $\lambda_{ii}
\,\neq\, 0$ and $\mu_{ii} \,\neq \, 0$ for all $i$.

Now we want to show that $\lambda_{ij}\, =\, 0 \,=\, \mu_{ij}$ for all $i \,\neq\, j$.

Without loss of generality, we may suppose that $i \,=\, 1$ and $j \,=\, 2$.

The Lie algebra $$S_{12}\ = \ \langle \exp (x_{1})(\lambda_{11}\partial_{x_1} + \cdots +\lambda_{1N}\partial_{x_N}),\,
\exp (-x_{1})(\mu_{11}\partial_{x_1} + \cdots +\mu_{1N}\partial_{x_N})\rangle$$
$$\times 
\langle \exp (x_{2})(\lambda_{21}\partial_{x_1} + \cdots +\lambda_{2N}\partial_{x_N}),\,
\exp (-x_{2})(\mu_{21}\partial_{x_1} + \cdots +\mu_{2N}\partial_{x_N})\rangle$$ is a Lie subalgebra of vector fields of the type
$$f_{1}(x_{1}, x_{2})\partial_{x_1} + f_{2}(x_{1}, x_{2})\partial_{x_2} + \cdots + f_{N}(x_{1}, x_{2})\partial_{x_N}.$$
Note that $S_{12}$ is isomorphic to $A_{1} \times A_{1}$. By Lemma \ref{lem2.1}, the linear map of vector fields defined by
\begin{equation}\label{r1}
f_{1}(x_{1}, x_{2})\partial_{x_1} + f_{2}(x_{1}, x_{2})\partial_{x_2} + \cdots + f_{N}(x_{1}, x_{2})\partial_{x_N}
\ \longmapsto\ f_{1}(x_{1}, x_{2})\partial_{x_1} + f_{2}(x_{1}, x_{2})\partial_{x_2}
\end{equation}
is a homomorphism of Lie algebras.

As $\lambda_{11}$, $\mu_{11}$, $\lambda_{22}$, $\mu_{22}$ are all not zero, the Lie algebra $S_{12}$ is mapped isomorphically, by the map
in \eqref{r1}, onto its image
$$\langle \exp (x_{1})(\lambda_{11}\partial_{x_1} + \lambda_{12}\partial_{x_2}), \
\exp (-x_{1})(\mu_{11}\partial_{x_1} + \mu_{12}\partial_{x_2})\rangle \times
$$
$$
\langle \exp (x_{2})(\lambda_{21}\partial_{x_1} + \lambda_{22}\partial_{x_2}), \
\exp (-x_{2})(\mu_{21}\partial_{x_1} + \mu_{22}\partial_{x_2})\rangle,$$
which is of type $A_1\times A_1$.

To prove the proposition, it therefore suffices to show that if
$${\mathfrak g} \,=\, \langle \exp (x)(\partial_{x} + \lambda \partial_{y}),\, \exp (-x)(\partial_{x} + \mu \partial_{y})\rangle
\times \langle \exp (y)(\partial_{y}
+ \widetilde{\lambda} \partial_{x}),\, \exp (-y)(\partial_{y} + \widetilde{\mu} \partial_{x})\rangle ,$$
where $\lambda$, $\mu$, $\widetilde{\lambda}$ and $\widetilde{\mu}$ are constants,
is a Lie algebra of type $A_{1} \times A_{1}$ on $\mathbb{C}^{2}$ with Cartan subalgebra $\langle \partial_{x},\, \partial_{y} \rangle$,
then
\begin{equation}\label{es}
\lambda\,=\,\mu\,=\, \widetilde{\lambda}\,=\, \widetilde{\mu}\,=\, 0.
\end{equation}

Now $[\exp (x)(\partial_{x} + \lambda \partial_{y}),\, \exp (-x)(\partial_{x} + \mu \partial_{y})]$ should be a non zero multiple of $\partial_{x}$.

Using the formula \eqref{1}, we see that $\lambda + \mu \ =\ 0$. Similarly, we have $\widetilde{\lambda} + \widetilde{\mu}\ =\ 0$.
Consequently, we have
$${\mathfrak g}\ =\ \langle \exp (x)(\partial_{x} + \lambda \partial_{y}),\, \exp (-x)(\partial_{x} - \lambda \partial_{y})\rangle \times
\langle \exp (y)(\partial_{y} + \widetilde{\lambda} \partial_{x}),\, \exp (-y)(\partial_{y} - \widetilde{\lambda} \partial_{x})\rangle .$$

As $[\exp (x)(\partial_{x} + \lambda \partial_{y}),\, \exp (y)(\partial_{y} + \widetilde{\lambda} \partial_{x})] \,=\, 0$
we see using the formula \eqref{1} that $\lambda(\partial_y + \widetilde{\lambda}\partial_x) - \widetilde{\lambda}(\partial_x +\lambda\partial_y)\,=\, 0$.
Therefore,
\begin{equation}\label{z1}
\lambda\ =\ \lambda\widetilde{\lambda} \ =\ \widetilde{\lambda}.
\end{equation}
On the other hand, as $[\exp (x)(\partial_{x} + \lambda \partial_{y}), \,
\exp (-y)(\partial_{y} -\widetilde{\lambda} \partial_{x})]\, = \,0$, from \eqref{1} it follows that
$$- \lambda(\partial_y - \widetilde{\lambda}\partial_x) + \widetilde{\lambda}(\partial_x +\lambda\partial_y)\,=\, 0.$$
So considering the coefficient of $\partial_x$ we conclude that $\lambda\widetilde{\lambda}+\widetilde{\lambda}\,=\, 0$. This and
\eqref{z1} together imply that
$\lambda \,=\, 0$ and $\widetilde{\lambda}\,=\, 0$. This proves \eqref{es}. As noted before, \eqref{es} completes the
proof of the proposition.
\end{proof}

\begin{lemma}\label{lem2}
Any simple Lie algebra which is not of type $A_\ell$ for some $\ell$ contains a subalgebra of type $B_2$, $G_2$ or $D_4$.
\end{lemma}

\begin{proof}
If the Dynkin diagram of a given simple Lie algebra $\mathfrak g$ contains a multiple bond,
then $\mathfrak g$ contains a subalgebra of type $B_2$ or $G_2$. Assume that there is no multiple bond
in the Dynkin diagram of $\mathfrak g$.

If ${\mathfrak g}$ is of type $D_n$, with $n\, \geq\, 4$, then it must contain a subalgebra of type $D_4$.

If $\mathfrak g$ is exceptional, it contains a subalgebra of type $E_6$
whose Dynkin diagram is
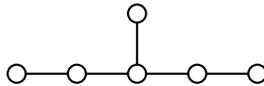
\begin{center}
\begin{tikzpicture}[scale=.4]
 \draw (-1,1) node[anchor=east] {$ $};
 \foreach \x in {0,...,4}
 \draw[thick,xshift=\x cm] (\x cm,0) circle (3 mm);
 \foreach \y in {0,...,3}
 \draw[thick,xshift=\y cm] (\y cm,0) ++(.3 cm, 0) -- +(14 mm,0);
 \draw[thick] (4 cm,2 cm) circle (3 mm);
 \draw[thick] (4 cm, 3mm) -- +(0, 1.4 cm);
 \end{tikzpicture}
\end{center}
Removing the two extreme vertices of the row we get the Dynkin diagram of $D_4$. So
$\mathfrak g$ contains a subalgebra of type $D_4$.
\end{proof}

Part (1) of the following proposition is due originally to Lie in the sense that it is a consequence
of Lie's classification of finite dimensional subalgebras of
vector fields on ${\mathbb C}^2$ \cite{Li}. This classification is also listed in
\cite[p. 369--372]{AH}, \cite[p. 3]{GKO} and \cite[pp.~472--475]{Olr}.

\begin{proposition}\label{lem3}\mbox{}
\begin{enumerate}
\item There is no faithful representation of a Lie algebra of type $B_2$ or $G_2$ as analytic
vector fields on ${\mathbb C}^2$.

\item There is no faithful representation of a Lie algebra of type $D_4$ as
analytic vector fields on ${\mathbb C}^4$.
\end{enumerate}
\end{proposition}

\begin{proof}
We will first show that there is no faithful representation of a Lie algebra of
type $B_2$ and $G_2$ in vector fields on ${\mathbb C}^2$.

To see this, if the positive roots of $B_2$ are $\alpha$, $\beta$, $\alpha +\beta$ and $\alpha +2\beta$, then
subalgebra of $B_2$ with roots
$$
\langle\alpha,\, \alpha+2\beta,\, -\alpha,\, -\alpha-2\beta\rangle
$$
is of type $A_1\times A_1$.

Similarly, if $G_2$ is a Lie subalgebra of Lie algebra of vector
fields on ${\mathbb C}^2$ and the positive roots are $\alpha$, $\beta$,
$\alpha +\beta$, $\alpha +2\beta$, $\alpha +3\beta$ and $2\alpha +3\beta$, then
the Lie subalgebra of $G_2$ with roots
$$
\langle \alpha, \, \alpha +2\beta,\, -\alpha,\, -\alpha -2\beta\rangle
$$
is of type $A_1\times A_1$. By Proposition \ref{prop2.4}, any Lie algebra of vector fields on ${\mathbb C}^2$ 
of type $A_1\times A_1$ is equivalent to
$$
S\, =\, \langle \exp (x)\partial_x,\, \exp (-x)\partial_x,\, \partial_x\rangle \times
\langle \exp (y)\partial_y,\, \exp (-y)\partial_y ,\, \partial_y \rangle.
$$
Therefore, $\langle \exp (x)\partial_x,\, \exp (y)\partial_y\rangle$ is a rank two abelian Lie subalgebra
in the derived algebra of a unique Borel subalgebra of $S$; this Borel subalgebra of $S$ is denoted by $\mathfrak b$.

Thus if $S$ is contained in a Lie subalgebra ${\mathfrak g}$ of vector
fields on ${\mathbb C}^2$, then
$\mathfrak g$ must have a highest weight vector $v$ in an $S$-invariant complement
of $S$ in $\mathfrak g$.

The derived algebra of the above Borel subalgebra $\mathfrak b$ of $S$ is $$\langle\exp (x)\partial_x,\ \exp (y)\partial_y\rangle .$$
In the coordinates $\widetilde{x},\, \widetilde{y}$ in which
$$
\exp (x)\partial_x\,=\, \partial_{\widetilde x}\ \ \, \text{ and }\ \ \,
\exp (y)\partial_y\,=\, \partial_{\widetilde y},
$$
we see that $[\partial_{\widetilde x},\, v]\,=\, 0\,=\, [\partial_{\widetilde y},\, v]$, as $v$ is
the above highest weight vector and a highest weight vector is, by definition, in the null space of the
derived algebra of a Borel subalgebra $\mathfrak b$. Now $v$ is of the form 
$$v=f(\widetilde{x},\widetilde{y})\partial_{\widetilde{x}}+g(\widetilde{x},\widetilde{y})\partial_{\widetilde{y}}.$$
Therefore,
$[\partial_{\widetilde{x}},\,v] \,=\,0\,=\, [\partial_{\widetilde{y}},\,v]$ implies that $f$ and $g$ are constant functions.
Hence
$$
v\, \in\, \langle \partial_{\widetilde x},\,\partial_{\widetilde y}\rangle,
$$
which is a contradiction.

Thus the only semisimple Lie algebras of vector fields on ${\mathbb C}^2$ are
of the type $A_1$, $A_1\times A_1$ and $A_2$.

Now we will prove that there is no faithful representation of a Lie algebra of type $D_4$ as vector fields on ${\mathbb C}^4$.

We note that if $\alpha,\, \beta,\, \gamma,\, \delta$ are the simple roots of $D_4$
$$
\begin{matrix}
&& \delta\\
&&\bullet\\
&& \vert\\
&& \vert\\
\bullet& ---& \bullet & --- &\bullet\\
\alpha && \beta && \gamma
\end{matrix}
$$
with $\alpha,\, \gamma,\, \delta$ orthogonal and
$$
\langle \alpha,\, \beta\rangle\,=\,\langle \beta,\, \gamma\rangle\,=\,
\langle \beta,\, \delta \rangle\,=\,-1,
$$
the highest root is $\alpha+2\beta+\gamma+\delta\,=:\, \alpha_0$ and $\alpha_0$ is orthogonal
to $\alpha,\, \gamma,\, \delta$, and thus the collection
$$
\{\alpha,\, \gamma,\, \delta,\, \alpha_0,\, -\alpha,\, -\gamma,\, -\delta,\, -\alpha_0\}
$$
corresponds to a Lie subalgebra of type $A_1\times A_1\times A_1\times A_1$. Its realization as vector fields on
${\mathbb C}^4$ is
$$
S\ =\ \prod_{i=1}^4 \langle \exp(x_i)\partial_{x_i},\, \exp(-x_i)\partial_{x_i},\, \partial_{x_i}\rangle ,
$$
with respect to some coordinates $x_1,\, x_2,\, x_3,\, x_4$,
and $$\langle \exp(x_1)\partial_{x_1},\, \exp(x_2)\partial_{x_2},\,
\exp(x_3)\partial_{x_3},\, \exp(x_4)\partial_{x_4}\rangle$$ is a rank $4$ abelian Lie subalgebra in the
derived algebra of a Borel subalgebra of $S$.
By the same argument as before, there is no highest weight vector in its complement in
$D_4$. This proves that
there is no faithful representation of $D_4$ in vector fields on ${\mathbb C}^4$.
\end{proof}

The following proposition is also due to Lie in the sense
that it is a consequence of Lie's classification of finite dimensional subalgebras of
vector fields on ${\mathbb C}^2$ \cite{Li}; see the lists in \cite[p. 369--372]{AH}, \cite[p. 3]{GKO}
and \cite[pp.~472--475]{Olr}.

\begin{proposition}\label{cor2}
All simple Lie algebras of vector fields on ${\mathbb C}^2$ must be of type $A_1$ or $A_2$.
\end{proposition}

\begin{proof}
By the main result of \cite{ABM}, the rank of such a Lie algebra can be at most two. Now Lemma \ref{lem2} and
Proposition \ref{lem3} together complete the proof.
\end{proof}

\section{Proof of Theorem \ref{thm1}}

By \cite[Theorem 2.1]{ABM}, if $\mathcal G$ is a semisimple Lie algebra of vector fields on ${\mathbb C}^N$, and
$C\, \subset\, {\mathcal G}$ is a Cartan subalgebra of dimension $n$ of $\mathcal G$, then there are
coordinates $x_1,\, \cdots,\, x_N$ so that the root spaces corresponding to a simple set of roots and their
negatives are of the form $\exp(x_i)V_i,\, \exp(-x_i)W_i$, $i\,=\, 1,\, \cdots, \, n$, where $V_i$ and $W_i$
are vector fields whose coefficients with respect to the basis $\partial_{x_1},\, \cdots,\, \partial_{x_N}$
are independent of $x_1,\, \cdots,\, x_n$ and the linear span of the vector fields $[\exp(x_i)V_i,\,
\exp(-x_i)W_i]$, $1\, \leq\, i\, \leq\, n$ is that of $\partial_{x_1},\, \cdots,\, \partial_{x_n}$,
and the Lie algebra generated by $\exp(x_i)V_i,\, \exp(-x_i)W_i$ is a copy of $sl(2,{\mathbb C})$ for
every $i\,=\, 1,\, \cdots, \, n$.

Thus in case the Cartan subalgebra is of dimension $N$, the vector fields $V_i,\, W_i$ are constant ones.
Consequently, if $\alpha_1,\, \cdots,\, \alpha_N$ are the simple roots, the corresponding root vectors are
vector fields of the form
\begin{equation}\label{e1}
X_{\alpha_i}\,=\, \exp(x_i)\left(\sum_{j=1}^N \lambda_{i,j}\partial_{x_j}\right),\ \
X_{-\alpha_i}\,=\, \exp(-x_i)\left(\sum_{j=1}^N \mu_{i,j}\partial_{x_j}\right),
\end{equation}
where $\lambda_{i,j}$ and $\mu_{i,j}$, $1\,\leq\, i,\, j\,\leq\, N$, are constants with $\lambda_{i,i}\, \not=\,
0\, \not=\, \mu_{i,i}$ for every $i$ as shown in the proof of Proposition \ref{prop2.4}.

Now suppose $J$ is a set of simple roots. By permuting $\alpha_1,\, \cdots,\, \alpha_N$ we may assume that
$$
J\,\,=\,\, \{\alpha_1,\, \cdots,\, \alpha_k\}.
$$
Note that $\{X_{\alpha_i},\, X_{-\alpha_i}\}_{i=1}^k$ (see \eqref{e1}) generate a semisimple Lie algebra
\begin{equation}\label{e2}
{\mathcal G}_J\, \subset\, {\mathcal G}.
\end{equation}
Moreover,
$${\mathfrak g}_{J}\ =\ {\mathfrak b}_{J}\oplus {\mathfrak c}_{J} \oplus {\mathfrak b}_{-J},$$
where ${\mathfrak c}_J$ is a Cartan subalgebra of ${\mathfrak g}_{J}$ and ${\mathfrak b}_{J}$, ${\mathfrak b}_{-J}$ are
Lie subalgebras obtained by taking commutators
of all possible order starting from the root vectors corresponding to the simple roots and their negatives.
The formula
$$
[\exp(\chi)V,\,\, \exp(\psi)W]\,=\, \exp(\chi+\psi)([V,\,\, W]+ V(\psi)W- W(\chi)V)
$$
gives that
$$
[\exp(\chi)V,\,\, \exp(\psi)W]\,=\, \exp(\chi+\psi)(V(\psi)\cdot W- W(\chi)\cdot V)
$$
if $V$ and $W$ are constant vector fields. Consequently, ${\mathcal G}_J$ in \eqref{e2} is
a Lie subalgebra of the Lie algebra of all vector fields of the type
$$
\sum_{i=1}^N f_i(x_1,\, \cdots,\, x_k) \partial_{x_i}\, ,
$$
and the projection
$$
\sum_{i=1}^N f_i(x_1,\, \cdots,\, x_k) \partial_{x_i}\, \,\longmapsto\,
\sum_{i=1}^k f_i(x_1,\, \cdots,\, x_k) \partial_{x_i}\, ,
$$
is a homomorphism of Lie algebras by Lemma \ref{lem2.1}. Consequently,
${\mathcal G}_J$ in \eqref{e2} is isomorphic to a Lie subalgebra of Lie algebra of vector fields on ${\mathbb C}^k$.

Suppose ${\mathcal G}_J$ is not of type $A_\ell$.

By Lemma \ref{lem2}, either ${\mathcal G}_J$ is of type $G_2$ or it has a Lie subalgebra of type
$B_2$ or $D_4$. Now applying Lemma \ref{lem2.1} we see that either we have a faithful representation of a Lie algebra of type $G_2$ or
$B_2$ as a Lie algebra of vector fields on ${\mathbb C}^2$ of rank two, or we have
a faithful representation of a Lie algebra of type $D_4$ as a Lie algebra of vector fields on ${\mathbb C}^4$ of rank four. But by
Proposition \ref{lem3} none of these representations exist. Thus every simple components
of $\mathcal G$ must of type $A_\ell$.

The faithful representations of $A_k$ in the Lie algebra of vector fields on ${\mathbb C}^k$ are already given in
Corollary 3.2 of \cite{ABM}. This describes all faithful representations of maximal rank in the Lie algebra of
vector fields on ${\mathbb C}^N$.

\section{An application}

A corollary of Theorem \ref{thm1} is the following; it is due originally to Lie in the sense that it can be obtained by inspection 
of the lists in \cite[p. 369--372]{AH} and \cite[p. 3]{GKO}.

Recall that every finite dimensional Lie algebra $\mathfrak g$ has a Levi decomposition ${\mathfrak g}\,=\, S\ltimes R$, where
$S$ --- the Levi complement --- is semisimple and $R$ --- the radical --- is a solvable ideal.

\begin{corollary}\label{cor1}
If $\mathfrak g$ is a Lie algebra of vector fields on ${\mathbb C}^2$ with a proper Levi decomposition,
then the Levi complement must be $sl(2,{\mathbb C})$.
\end{corollary}

\begin{proof}
Let ${\mathfrak g}\, \subset\, V(\mathbb{C}^{2})$ be a finite dimensional Lie algebra of vector fields on ${\mathbb C}^2$
with Levi decomposition
$${\mathfrak g}\ =\ S\ltimes R .$$ Assume that both $S$ and $R$ are not zero.

By adopting the arguments given in the proof of \cite[Theorem 2.1]{ABM}, as all Cartan subalgebras of a complex
semisimple Lie algebra are split, we see that for any complex semisimple Lie algebra
of vector fields on $\mathbb{C}^{N}$, its Cartan subalgebra can be of dimension at most $N$.

Thus $S$ is of rank 1 or 2. If S is of rank 
2, then it is of type $A_1\times A_1$ or $A_{2}$ --- by Proposition \ref{cor2}. We have already seen that any Lie
algebra of type $A_1\times A_1$ in $\mathbb{C}^{2}$ contains the abelian Lie algebra $\langle 
\exp(x)\partial_{x},\exp(y)\partial_{y} \rangle$ --- for suitable local coordinates $x,\, y$ --- in the derived algebra of a Borel subalgebra.

A Lie algebra of type $A_{2}$ is isomorphic to $sl(3,\mathbb{C})$ and the Lie subalgebra of upper triangular matrices
is a Borel subalgebra of $sl(3,\mathbb{C})$. The derived algebra of this Borel subalgebra of
$sl(3, {\mathbb C})$ has generators $X,\, Y,\, Z$ such that $[X,\,Y]\,=\,Z$
with $Z$ commuting with both $X$ and $Y$. Choose coordinates in which $Z\,=\,\partial_{x}$.
Thus $X\,=\, f_{1}(y)\partial_{x}+g_{1}(y)\partial_{y}$
and $Y\,=\,f_{2}(y)\partial_{x}+g_{2}(y)\partial_{y}$. If both the Lie algebras $\langle Z, \,X \rangle$ and $\langle Z,\, Y \rangle$
are of rank 1, say 
$X\,=\,f_{1}(y)\partial_{x}$,\, $Y\,=\,f_{2}(y)\partial_{x}$, then we have
$[X,\, Y]\,=\,0$. Thus, one of $\langle Z,\, X\rangle$ or $\langle Z,\, Y\rangle$ must be of rank 2.
Consequently, the derived algebra of a Borel subalgebra of both the Lie algebras of type $A_1\times A_1$ and $A_2$ contains an abelian
Lie algebra of rank 2. In the canonical coordinates $x,y$ of this abelian Lie
algebra, any highest weight vector in $V(\mathbb{C}^{2})$ must be in the Lie algebra $\langle \partial_{x},\, \partial_{y} \rangle$.

Thus the Levi complement $S$ must be of type $A_{1}$. Therefore, the only Lie algebras on $\mathbb{C}^{2}$ with a proper Levi
decomposition must be of the form $sl(2,\mathbb{C})\rtimes R$, where $R$ is the radical of the Lie algebra.
\end{proof}

Theorem \ref{thm1} reduces the classification of semisimple Lie algebras of vector fields on ${\mathbb 
C}^3$ and those Lie algebras with a proper Levi decomposition to the representations of ranks one and two 
in the vector fields on ${\mathbb C}^3$; proofs of this classification will appear elsewhere.
(Here by rank we mean the dimension of a Cartan subalgebra of a 
given semisimple Lie algebra.)

\section*{Acknowledgements}

We are very grateful to Karl Hermann Neeb and the referee whose detailed comments have greatly improved the
readability of the paper. The second-named author is 
partially supported by the J. C. Bose Fellowship JBR/2023/000003.

\section*{Mandatory declarations}

The authors have no conflict of interest to declare that are relevant to this article.
No data were generated or used.


\begin{thebibliography}{ZZZZZZ}

\bibitem[AH]{AH} M. Ackerman and R. Hermann, {\it Sophus Lie's 1880 transformation group paper}, Lie Groups: History, Frontiers
and Applications, Vol. 1, Math Sci Press, 18 Gibbs Street, Brookline, MA 02146, 1976.

\bibitem[ABM]{ABM} H. Azad, I. Biswas and F. M. Mahomed, Equality of the algebraic and geometric ranks of
Cartan subalgebras and applications to linearization of a system of ordinary differential
equations, {\it Internat. J. Math.} {\bf 28}, no. 11, (2017).

\bibitem[Lee]{Lee}
J. M. Lee, \textit{Manifolds and Differential Geometry}, 
vol. 107, American Mathematical Society, 2009.

\bibitem[GKO]{GKO} A. Gonz\'alez-L\'opez, N. Kamran and P. J. Olver, Lie algebras of vector fields in
the real plane, {\it Proc. London Math. Soc.} {\bf 64} (1992), 339--368.

\bibitem[Ki]{Ki} A. Kirillov, {\it An introduction to Lie groups and Lie algebras},
Cambridge Stud. Adv. Math., 113, Cambridge University Press, Cambridge, 2008.

\bibitem[Li]{Li} S. Lie, Theorie der Transformationsgruppen I, {\it Math. Ann.} {\bf 16} (1880), 441--528.

\bibitem[Ol]{Olr} P. J. Olver, \textit{Equivalence, invariants, and symmetry},
Cambridge University Press, Cambridge, 1995.

\bibitem[Vi]{Vi} \`E. B. Vinberg, Algebraic transformation groups of maximal rank, {\it Mat. Sb.}
{\bf 88} (1972), 493--503.

\end{thebibliography}
\end{document}